\documentclass[preprint,12pt]{elsarticle}

\makeatletter
\def\ps@pprintTitle{%
 \let\@oddhead\@empty
 \let\@evenhead\@empty
 \def\@oddfoot{}%
 \let\@evenfoot\@oddfoot}
\makeatother

\let\svthefootnote\thefootnote
\newcommand\blankfootnote[1]{%
  \let\thefootnote\relax\footnotetext{#1}%
  \let\thefootnote\svthefootnote%
}
\let\svfootnote\footnote
\renewcommand\footnote[2][?]{%
  \if\relax#1\relax%
    \blankfootnote{#2}%
  \else%
    \if?#1\svfootnote{#2}\else\svfootnote[#1]{#2}\fi%
  \fi
}

\usepackage{amsmath,amssymb,amsthm}
\usepackage{mathtools} % used for psmallmatrix

\usepackage{float}  

\usepackage{hyperref} %Internal and external links
\hypersetup{breaklinks, colorlinks}

\usepackage{natbib}

\usepackage[text={6in,8.8in},centering]{geometry}

\DeclareSymbolFont{bbold}{U}{bbold}{m}{n}
\DeclareMathSymbol{\bbone}{\mathord}{bbold}{49}    % mathbb 1
\DeclareMathSymbol{\weight}{\mathord}{bbold}{"77}  % mathbb w

\theoremstyle{plain}% default
\newtheorem{theorem}{Theorem}[section]
\newtheorem{corollary}{Corollary}[section]
\newtheorem{lemma}{Lemma}[section]
\newtheorem{proposition}{Proposition}[section]

\newtheoremstyle{bodyfontup}
    {\dimexpr\topsep/2\relax} % space above
    {\dimexpr\topsep/2\relax} % space below
    {}          % body font
    {}          % indent amount
    {\bfseries} % theorem head font
    {.}         % punctuation after theorem head
    {.5em}      % space after theorem head
    {}          % theorem hed spec. (empty = "normal")
    
\theoremstyle{bodyfontup}
\newtheorem{remark}{Remark}[section]
\newtheorem{example}{Example}[section]

\makeatletter
\renewenvironment{proof}[1][\proofname]{%
   \par\pushQED{\qed}\normalfont%
   \topsep6\p@\@plus6\p@\relax
   \trivlist\item[\hskip\labelsep\bfseries#1\@addpunct{.}]%
   \ignorespaces
}{%
   \popQED\endtrivlist\@endpefalse
}
\makeatother

\newcommand{\CC}{\mathbb{C}}
\newcommand{\RR}{\mathbb{R}}
\newcommand{\NN}{\mathbb{N}}
\newcommand{\ZZ}{\mathbb{Z}}
\newcommand{\Zpl}{\ZZ_+}

\newcommand{\abs}[1]{\lvert#1\rvert}           % absolute value 
\newcommand{\Abs}[1]{\big\lvert#1\big\rvert}   % absolute value (big)
\newcommand{\ABS}[1]{\Big\lvert#1\Big\rvert}   % absolute value (Big)

\newcommand{\card}[1]{\abs{#1}}                % cardinality of a set

\newcommand{\coeff}[2]{\mathrm{Coeff}(#1,#2)}
\newcommand{\Coeff}[2]{\mathrm{Coeff}\Bigl(#1,#2\Bigr)}

\renewcommand{\geq}{\geqslant}

\renewcommand{\leq}{\leqslant}

\newcommand{\norm}[1]{\|#1\|}

\newcommand{\per}{\mathrm{per}}  % permanent of square matrix

\newcommand{\set}[1]{\underline{#1}}

\newcommand{\vecsum}[1]{\abs{#1}} 

\newcommand{\binomial}[2]{\genfrac{(}{)}{0pt}{}{#1}{#2}}

\newcommand{\newatop}[2]{\genfrac{}{}{0pt}{}{\scriptstyle #1}{\scriptstyle #2}}

\scrollmode
\begin{document}
%%%%%%%%%%%%%%%%%%%%%%%%%%%%%%%%%%%%%%%%%%%%%%%%%%%%%%%%%%%%%%%%%%%%%%
\begin{frontmatter}

\title{Generalization of a Hadamard type \\
inequality for 
permanents\texorpdfstring{\footnote[]{\textcopyright\ 2019. This 
manuscript version is made available under the CC-BY-NC-ND 4.0 
license \\
\url{https://creativecommons.org/licenses/by-nc-nd/4.0/}}}{}
\texorpdfstring{\footnote[]{The published version is available at 
\url{https://doi.org/10.1016/j.laa.2019.02.015}}}{}}

\author{Bero Roos}

\ead{bero.roos@uni-trier.de}

\address{FB IV -- Mathematics, 
University of Trier,
54286 Trier,
Germany\medskip\\ \normalsize \textup{February 27, 2019}}

\begin{abstract}
This paper is devoted to a generalization of a Hadamard type 
inequality for the permanent of a complex square matrix. 
Our proof is based on a non-trivial extension of a technique
used in Carlen, Lieb and Loss (Methods and Applications of Analysis 
13 (1) (2006) 1--17). We give an application to coefficients of 
products of linear forms and show some auxiliary inequalities, 
which might be of independent interest. 
\end{abstract}


\end{frontmatter}

%%%%%%%%%%%%%%%%%%%%%%%%%%%%%%%%%%%%%%%%%%%%%%%%%%%%%%%%%%%%%%%%%%%%%%
\section{Introduction and the main result}
%%%%%%%%%%%%%%%%%%%%%%%%%%%%%%%%%%%%%%%%%%%%%%%%%%%%%%%%%%%%%%%%%%%%%%
In a paper of 1893, \citet{hadamard1893} proved an inequality for 
the determinant of a complex $n\times n$ matrix $Z=(z_{j,r})$, which 
states that
\begin{align}\label{e5174597}
\abs{\det(Z)}
\leq \prod_{r=1}^n\Bigl(\sum_{j=1}^n\abs{z_{j,r}}^2\Bigr)^{1/2},
\end{align}
that is, the modulus of the determinant of $Z$ is bounded from above 
by the product of the Euclidean norms of all column vectors of 
$Z$. Nowadays, there are several comparatively simple proofs of 
\eqref{e5174597} available because of the nice properties of the 
determinant, e.g.\ see \citet{MR2978290}. 
The permanent, however, defined as a kind of ``sign-less'' 
determinant, does not share many properties of the determinant, 
such as the compatibility with matrix multiplication. 
Therefore, it is not so surprising that a Hadamard type inequality 
for permanents was found much later, in 2006, see \citet{MR2275869} 
and~\citet{MR2256997}.

To be more precise, we need some notation. 
For arbitrary sets $A$ and $B$, let $A^B$ denote the set of all maps 
$f:\,B\longrightarrow A$ and define 
$A_{\neq}^B=\{f\in A^B\,|\, f \mbox{ is injective}\}$. 
If $f\in A^B$ and $b\in B$, we write $f_b=f(b)$. 
For $n\in\NN=\{1,2,3,\dots\}$, let $\set{n}=\{1,\dots,n\}$, 
$A^n=A^{\set{n}}=\{(a_1,\dots,a_n)\,|\,a_1,\dots,a_n\in A\}$,
and $A_{\neq}^n=A_{\neq}^{\set{n}}$. In particular, 
$\set{n}_{\neq}^n$ is the set of all permutations on $\set{n}$. 
Let $\RR$ (resp.\ $\CC$) be the set of all real (resp.\ complex) 
numbers. 
If two sets $J$ and $R$ have the same finite cardinality 
$\card{J}=\card{R}\in\Zpl=\NN\cup\{0\}$, the permanent $\per(Z)$ of 
a matrix $Z=(z_{j,r})\in\CC^{J\times R}$ can be defined as
the sum of all diagonal products of $Z$, i.e.\ $\per(Z)$ is
the sum of products $\prod_{j\in J}z_{j,r_j}$ over all 
$r\in R_{\neq}^J$. 
In the special case $J=R=\emptyset$, this gives $\per(Z)=1$,
since empty products are defined as~$1$. 
For the sets $J'\subseteq J$ and $R'\subseteq R$, 
let $Z[J',R']\in \CC^{J'\times R'}$ denote the submatrix of $Z$ 
with entries $z_{j,r}$ for $(j,r)\in J'\times R'$. 
Therefore, if $n\in\NN$, 
$Z=(z_{j,r})\in\CC^{n\times n}=\CC^{\set{n}\times\set{n}}$, and 
$J,R\subseteq\set{n}$ with $\card{J}=\card{R}$, we have
\begin{align*}
\per(Z)=\sum_{r\in\set{n}_{\neq}^n}\prod_{j=1}^nz_{j,r_j}
=\sum_{j\in\set{n}_{\neq}^n}\prod_{r=1}^nz_{j_r,r}, \quad
\per(Z[J,R])
=\sum_{r\in R_{\neq}^J}\prod_{j\in J}z_{j,r_j}
=\sum_{j\in J_{\neq}^R}\prod_{r\in R}z_{j_r,r}.
\end{align*}
An overview of properties and applications of permanents is provided 
in \citet{MR504978,MR688551,MR900069} and \citet{MR2140290}. 
For two non-empty sets $J$, $R$ and $Z=(z_{j,r})\in\CC^{J\times R}$, 
let 
\begin{align*}
Q(Z)=\prod_{r\in R}\Bigl(\sum_{j\in J}\abs{z_{j,r}}^2\Bigr)^{1/2}
\end{align*}
be the product of the Euclidean norms of all column vectors of $Z$.

Let us come back to the remarkable Hadamard type inequality for 
permanents. This says that, for $n\in\NN$ and
$Z=(z_{j,r})\in\CC^{n\times n}$, 
\begin{align}\label{e75870}
\abs{\per(Z)}\leq n!\prod_{r=1}^n\Bigl(\frac{1}{n}\sum_{j=1}^n
\abs{z_{j,r}}^2\Bigr)^{1/2}.
\end{align}
In other words, the modulus of the permanent of $Z$ is bounded from 
above by $\frac{n!}{n^{n/2}} Q(Z)$.

There are a few methods of proof of \eqref{e75870}.
\citet[Theorem 1.1]{MR2275869} gave two different proofs,
the first one of which uses a heat kernel interpolation argument,
while the second one is elementary and is based on induction and the 
arithmetic-geometric mean inequality. 
Furthermore, they determined all cases of equality in \eqref{e75870}
for $n\geq 3$,
see also Remark \ref{r62875}(\ref{r62875.d}) below.
Another proof can be found in \citet[Theorem 5.1]{MR2256997},
who used a Hilbert space technique.
As stated in \cite[Introduction]{MR2275869}, \eqref{e75870}
can also be obtained from Theorem~9.1.1 in Appendix~1 of  
\citet{MR1258086}. 
%%%%%%%%%%%%%%%%%%%%%%%%%%%%%%%%%%%%%%%%%%%%%%%%%%%%%%%%%%%%%%%%%%%%%%

\newpage
In their second proof of \eqref{e75870}, 
\citet[page 12]{MR2275869} showed that, 
for  $\emptyset\neq L\subseteq\set{n}$, $s\in L$, and 
$\ell=\card{L}$, 
\begin{align}\label{e37876}
\sum_{\newatop{J\subseteq\set{n}}{\card{J}=\ell}}
\Abs{\per(Z[J,L])}^2
&\leq \ell(n-\ell+1) 
\Bigl(\frac{1}{n}\sum_{j\in\set{n}}\abs{z_{j,s}}^2\Bigr)
\sum_{\newatop{J\subseteq\set{n}}{\card{J}=\ell-1}}
\Abs{\per(Z[J,L\setminus\{s\}])}^2,
\end{align}
from which they obtained inductively their Theorem 3.1, that is
\begin{align} \label{e738659}
\sum_{\newatop{J\subseteq\set{n}}{\card{J}=\ell}}
\Abs{\per(Z[J,L])}^2
&\leq (\ell!)^2\binomial{n}{\ell}\prod_{r\in L}\Bigl(\frac{1}{n}
\sum_{j\in\set{n}}\abs{z_{j,r}}^2\Bigr).
\end{align}
To put it differently, the left-hand side of \eqref{e738659} is 
bounded  by 
$\frac{(\ell!)^2}{n^{\ell}}\binomial{n}{\ell}Q(Z[\set{n},L])$. 
For $L=\set{n}$ and $\ell=n$, this reduces to inequality 
\eqref{e75870}. In the present paper, we present generalizations of 
\eqref{e37876} and \eqref{e738659} and, in turn, of \eqref{e75870}. 
We note that \citet[Theorem 1]{MR0170901} obtained an upper bound for 
the expression 
$\sum_{\newatop{J,L\subseteq\set{n}}{\card{J}=\card{L}=\ell}}
\Abs{\per(Z[J,L])}^2$,  which can also 
be estimated by using \eqref{e738659} or \eqref{e84376} below. The 
resulting bounds, however, are not easily comparable with the one 
in~\cite{MR0170901}.

Our first result is the following theorem, which is a direct 
consequence of the more general Theorem \ref{c628560} below.

%%%%%%%%%%%%%%%%%%%%%%%%%%%%%%%%%%%%%%%%%%%%%%%%%%%%%%%%%%%%%%%%%%%%%%

\begin{theorem} \label{t762986}
Let $n\in\NN$, $d\in\set{n}$, 
$M_1,\dots,M_d\subseteq \set{n}$ be pairwise disjoint sets with 
$\bigcup_{k=1}^dM_k=\set{n}$ and $m_k=\card{M_k}\in\NN$ 
for $k\in\set{d}$.
Let $Z=(z_{j,r})\in\CC^{n\times n}$ and 
$\alpha_{j,k}=\frac{1}{m_k}\sum_{r\in M_k}\abs{z_{j,r}}^2$
for $j\in\set{n}$ and $k\in\set{d}$. 
Then 
\begin{align}\label{e73766509}
\abs{\per(Z)}
&\leq n! \prod_{k=1}^d\Bigl(\frac{1}{\binomial{n}{m_k}}
\sum_{\newatop{I\subseteq \set{n}}{\card{I}=m_k}}
\prod_{j\in I}\alpha_{j,k}\Bigr)^{1/2}. 
\end{align}
\end{theorem}
%%%%%%%%%%%%%%%%%%%%%%%%%%%%%%%%%%%%%%%%%%%%%%%%%%%%%%%%%%%%%%%%%%%%%%
\begin{remark}\phantomsection\label{r62875}
% \phantomsection is used for correct links referring to Remark 1.1
\begin{enumerate}[(a)] 

\item
The modulus of the permanent of $Z$ is bounded from above 
by $n!\prod_{k=1}^d\bigl((\binomial{n}{m_k}m_k^{m_k})^{-1}
\sum_{\newatop{I\subseteq \set{n}}{\card{I}=m_k}}
Q(Z^T[M_k,I])^2\bigr)^{1/2}$, where $Z^T$ is the transpose of~$Z$. 

\item 
If $d=n$ and $m_1=\dots=m_n=1$, then 
\eqref{e73766509} reduces to \eqref{e75870}. 

\item  
If $d=1$ and $m_1=n$, then \eqref{e73766509} reduces to 
\eqref{e75870} applied to the transpose of $Z$.

\item \label{r62875.d}
It is easily shown that, in \eqref{e73766509}, equality holds, if 
one of the following conditions is true: 
\begin{itemize} 

\item a number $k\in\set{d}$ exists such that 
$\card{\{j\in\set{n}\,|\,z_{j,r}=0\mbox{ for all }r\in M_k\}}
\geq n-m_k+1$
or

\item there are numbers $\xi_j,\zeta_j\in\CC$ with 
$\abs{\xi_j}=\abs{\zeta_j}=1$ for 
all $j\in\set{n}$  and $y_{k}\in(0,\infty)$ for all $k\in\set{d}$,
such that $z_{j,r}=\xi_{j}\zeta_r y_{k}$
for all $j\in\set{n}$, $k\in\set{d}$ and $r\in M_k$.

\end{itemize}
\newpage
As was shown by \citet[Theorem 1.1]{MR2275869}, for $n=d\geq 3$ and 
$m_1=\dots=m_n=1$, equality in \eqref{e73766509} is equivalent to 
the condition above. However, in general, the situation may be more 
complicated as is outlined in the next example. 

\end{enumerate}
\end{remark}

%%%%%%%%%%%%%%%%%%%%%%%%%%%%%%%%%%%%%%%%%%%%%%%%%%%%%%%%%%%%%%%%%%%%%%
\begin{example}\label{ex167598}
Let the assumptions of Theorem \ref{t762986} hold.
\begin{enumerate}[(a)]

\item \label{ex167598.a}
Let us consider the case 
$n=d=2$, $M_1=\{1\}$, $M_2=\{2\}$, and $m_1=m_2=1$. 
Then \eqref{e73766509} says that 
$\abs{\per(Z)}^2
\leq (\abs{z_{1,1}}^2+\abs{z_{2,1}}^2)
(\abs{z_{1,2}}^2+\abs{z_{2,2}}^2)$,
which can also be shown directly using the Cauchy-Schwarz inequality:
\begin{align} 
\begin{split}
\abs{\per(Z)}^2
&=\abs{z_{1,1}z_{2,2}+z_{1,2}z_{2,1}}^2
\leq (\abs{z_{1,1}z_{2,2}}+\abs{z_{1,2}z_{2,1}})^2 \\
&\leq (\abs{z_{1,1}}^2+\abs{z_{2,1}}^2)(\abs{z_{1,2}}^2
+\abs{z_{2,2}}^2).
\end{split}\label{eq987657}
\end{align}
It is easily shown that, in the chain \eqref{eq987657}, equality 
holds if and only if $\abs{z_{1,1}z_{1,2}}=\abs{z_{2,1}z_{2,2}}$
and there is a number $\xi\in\CC$ such that  
$\abs{\xi}=1$, $z_{1,1}z_{2,2}=\xi\abs{z_{1,1}z_{2,2}}$, and
$z_{1,2}z_{2,1}=\xi\abs{z_{1,2}z_{2,1}}$. It is clear that the 
very left and right sides of \eqref{eq987657} are equal to zero if 
and only if $z_{1,1}=z_{2,1}=0$ or $z_{1,2}=z_{2,2}=0$. But equality 
in \eqref{eq987657} also holds if $z_{1,1}=z_{2,2}=0$ or 
$z_{1,2}=z_{2,1}=0$. We note that the case of equality in 
\eqref{e73766509} under the present assumptions was not discussed in 
\citet{MR2275869}. 

\item \label{ex167598.b} 
We now investigate the cases of equality in 
\eqref{e73766509} for $n=3$, $d=2$, $M_1=\set{2}$, $M_2=\{3\}$, 
$m_1=2$, and $m_2=1$. 
Let $e=(z_{1,1},z_{1,2})$, $f=(z_{2,1},z_{2,2})$, and
$g=(z_{3,1},z_{3,2})$ denote the row vectors of 
$Z[\set{3},\set{2}]$ and 
$h=(z_{1,3},z_{2,3},z_{3,3})^T$ the last 
column vector of $Z$, where $T$ means transposition as previously. 
Inequality \eqref{e73766509} is equivalent to 
\begin{align} \label{eq357865}
\abs{\per(Z)}^2
&\leq (\norm{e}^2\norm{f}^2+\norm{e}^2\norm{g}^2+\norm{f}^2\norm{g}^2)
\norm{h}^2,
\end{align}
where $\norm{\cdot}$ denotes the Euclidean norm. 
\begin{enumerate}[(i)] 

\item 
If $h=0$ or at least two of the vectors $e$, $f$, or $g$ are zero, 
then, in \eqref{eq357865}, equality holds, since both sides are 
equal to zero. 

\item 
Let us assume that $h\neq0$ and that 
exactly one of the vectors $e$, $f$, $g$ is equal to zero. 
For simplicity, we assume that $e=0$, $f\neq0$, and $g\neq0$. 
The remaining cases are treated analogously. 
Under the assumptions above, \eqref{eq357865} reduces to the chain
\begin{align} 
\abs{\per(Z)}^2
=\abs{\per(Z[\{2,3\},\set{2}])}^2\, \abs{z_{1,3}}^2
&\leq\norm{f}^2\norm{g}^2\norm{h}^2.\label{eq8365}
\end{align}
Using Part (\ref{ex167598.a}) of this example applied to the 
transpose of $Z[\{2,3\},\set{2}]$, we see that equality in 
\eqref{eq8365} holds if and only if  
$z_{1,3}\neq 0$, 
$z_{2,3}=z_{3,3}=0$,
$\abs{z_{2,1}z_{3,1}}=\abs{z_{2,2}z_{3,2}}$,
and there is a number $\xi\in\CC$ such that 
$\abs{\xi}=1$, $z_{2,1}z_{3,2}=\xi\abs{z_{2,1}z_{3,2}}$, 
$z_{3,1}z_{2,2}=\xi\abs{z_{3,1}z_{2,2}}$. 
We note that, if $z_{2,1}=0$, then $z_{2,2}\neq0$, $z_{3,2}=0$, 
and $z_{3,1}\neq0$; similarly, if $z_{2,2}=0$, then $z_{2,1}\neq0$, 
$z_{3,1}=0$, and $z_{3,2}\neq0$. 

\item 
Let us now assume that $e$, $f$, $g$, and $h$ are all non-zero 
vectors. Then, in \eqref{eq357865}, equality holds if and only if 
$Z\in(\CC\setminus\{0\})^{3\times 3}$ and 
there are numbers $\xi_j,\zeta_r\in\CC$ for $j,r\in\set{3}$ and 
$x\in(0,\infty)$, such that $\abs{\xi_j}=\abs{\zeta_r}=1$,
$z_{j,r}=\xi_{j}\zeta_r \abs{z_{j,r}}$, 
$\abs{z_{j,1}}=\abs{z_{j,2}}$, and
$\abs{z_{j,3}}=x\prod_{i\in\set{3}\setminus\{j\}}\abs{z_{i,1}}$
for all $j,r\in\set{3}$.
It is easily shown that the condition above is indeed sufficient
for equality in \eqref{eq357865}. The necessity is proved in 
Section~\ref{s7296} below.
\end{enumerate}

\item Let the assumptions of Part (\ref{ex167598.b}) be valid. 
From \eqref{e75870} we obtain 
\begin{align}\label{e75871}
\abs{\per(Z)}^2
\leq 36\prod_{r=1}^3
\Bigl(\frac{1}{3}\sum_{j=1}^3\abs{z_{j,r}}^2\Bigr).
\end{align}
Let us assume that $Z\in[0,\infty)^{3\times 3}$ and that $h\neq0$. 
Inequality \eqref{eq357865} is better than \eqref{e75871} if and 
only if 
$W:=(\norm{e}^2\norm{f}^2+\norm{e}^2\norm{g}^2+\norm{f}^2\norm{g}^2)
-\frac{4}{3}\prod_{r=1}^2\bigl(\sum_{j=1}^3\abs{z_{j,r}}^2\bigr)<0$.
It is easily seen that
\begin{align*}
W&=\sum_{j\in J}\Bigl(\frac{1}{3}
(z_{j_1,1}^2-z_{{j_2},1}^2)(z_{j_1,2}^2-z_{{j_2},2}^2)
-(z_{j_1,1}^2-z_{{j_2},2}^2)(z_{j_1,2}^2-z_{{j_2},1}^2)\Bigr),
\end{align*}
where $J=\{(1,2),(1,3),(2,3)\}$. 
Sometimes, $W$ can indeed be negative: for instance, 
if $z_{j,1}=z_{j,2}$ for all $j\in\set{3}$, then we have 
$W=-\frac{2}{3}\sum_{j\in J}(z_{j_1,1}^2-z_{{j_2},1}^2)^2\leq0$;
if $z_{1,1}=z_{1,2}=z_{2,2}=z_{3,1}=0$, then 
$W=-\frac{1}{3}z_{2,1}^2z_{3,2}^2\leq0$.  
However, if $z_{1,r}=z_{2,r}=z_{3,r}$ for all $r\in\set{2}$, 
then $W=3(z_{1,1}^2-z_{1,2}^2)^2\geq0$. 
\end{enumerate}

\end{example}

%%%%%%%%%%%%%%%%%%%%%%%%%%%%%%%%%%%%%%%%%%%%%%%%%%%%%%%%%%%%%%%%%%%%%%
\noindent
Under  certain assumptions, \eqref{e73766509} can be simplified, 
as is shown next.
%%%%%%%%%%%%%%%%%%%%%%%%%%%%%%%%%%%%%%%%%%%%%%%%%%%%%%%%%%%%%%%%%%%%%%
\begin{corollary} \label{c639860}
Let the assumptions of Theorem \textup{\ref{t762986}} be valid. 
For $k\in\set{d}$, let $s_k\in M_k$ be fixed. Let us assume that
$\abs{z_{j,r}}=\abs{z_{j,s_k}}$ for all $j\in\set{n}$, $k\in\set{d}$, 
and $r\in M_k$. Then   
\begin{align}\label{e52447}
\abs{\per(Z)}
&\leq n! \prod_{k=1}^d\Bigl(\frac{1}{\binomial{n}{m_k}}
\sum_{\newatop{I\subseteq \set{n}}{\card{I}=m_k}}
\prod_{j\in I}\abs{z_{j,s_k}}^2\Bigr)^{1/2}. 
\end{align}
\end{corollary}
%%%%%%%%%%%%%%%%%%%%%%%%%%%%%%%%%%%%%%%%%%%%%%%%%%%%%%%%%%%%%%%%%%%%%%
\begin{remark}\label{r6386}
The right-hand sides of \eqref{e73766509}
and \eqref{e52447} can be further estimated using Maclaurin's 
inequality, which says that, for $n\in\NN$ and
$y_1,\dots,y_n\in[0,\infty)$, the normalized elementary symmetric 
polynomials 
$S_m=\frac{1}{\binomial{n}{m}}
\sum_{{I\subseteq{\set{n}}:\,}{\card{I}=m}}\prod_{j\in I}y_j$
satisfy $(S_{m})^{1/m}\geq (S_{m+1})^{1/(m+1)}$ for all 
$m\in\set{n-1}$; see \citet[Theorem 52, page 52]{MR0046395}. 
Hence, under the assumptions of Corollary \ref{c639860},
\eqref{e52447} is a sharpening of \eqref{e75870}. 
\end{remark}
%%%%%%%%%%%%%%%%%%%%%%%%%%%%%%%%%%%%%%%%%%%%%%%%%%%%%%%%%%%%%%%%%%%%%%

The rest of the paper is structured as follows. 
The next section is devoted to an application of Theorem~\ref{t762986}
to the coefficients of products of linear forms.
In Section~\ref{s65275}, we present and prove 
Theorem~\ref{c628560}, which generalizes Theorem~\ref{t762986}. 
The proof requires a technical inequality stated in
Proposition~\ref{l845097}, the proof of which we defer to  
Section~\ref{s376598}. Additionally, the cases of equality are 
determined. Section~\ref{s7296} contains the remaining 
proofs. 
%%%%%%%%%%%%%%%%%%%%%%%%%%%%%%%%%%%%%%%%%%%%%%%%%%%%%%%%%%%%%%%%%%%%%%
\section{Application to coefficients of products of linear forms}

%%%%%%%%%%%%%%%%%%%%%%%%%%%%%%%%%%%%%%%%%%%%%%%%%%%%%%%%%%%%%%%%%%%%%%
It is well-known that permanents are certain coefficients of 
products of linear forms. 
More precisely, for $n\in\NN$ 
and $Z=(z_{j,r})\in\CC^{n\times n}$, we have  
\begin{align}\label{e78365}
\per(Z)
=\Coeff{x_1\cdots x_n}{\prod_{j=1}^n\Bigl(\sum_{r=1}^nz_{j,r}x_r
\Bigr)},
\end{align}
e.g.\ see \citet[page 103]{MR504978}.
Here, $a_m=\coeff{x^m}{f(x)}$ denotes the coefficient of 
$x^m=\prod_{j=1}^nx_j^{m_j}$ in the (terminating) formal power series 
$f(x)=\sum_{k\in\Zpl^n}a_kx^k$ with $a_k\in\CC$,  
where $m\in\Zpl^n$, $x_1,\dots,x_n$ are algebraically independent
commuting indeterminates over $\CC$, and we write $x=(x_1,\dots,x_n)$.
On the other hand, some expressions more general than 
\eqref{e78365} can be represented  in terms of permanents.
This is shown in the next lemma, which goes back to \citet{muir1912},
who considered the case $n=d=4$. A proof for $n=d\in\NN$ can be 
found in \citet[formula (2.5)]{MR671835} and easily generalizes to 
$n,d\in\NN$. Since we use a different notation and to make the 
present paper more self-contained, we provide another short proof in 
Section \ref{s7296}.

%%%%%%%%%%%%%%%%%%%%%%%%%%%%%%%%%%%%%%%%%%%%%%%%%%%%%%%%%%%%%%%%%%%%%%

We need the following notation. 
For $d\in\NN$ and $m=(m_1,\dots,m_d)\in\Zpl^d$, 
set $\vecsum{m}=\sum_{k=1}^d m_k$ and $m!=\prod_{k=1}^dm_k!$. 
Furthermore, for a set $K\neq \emptyset$, $n\in\NN$, $s\in K^n$, 
and $k\in K$, let $\weight_k(s)=\sum_{j=1}^n\bbone_{\{k\}}(s_j)$
be the number of $k$'s in the vector $s$. 
Here, for a set $A$, let $\bbone_{A}(x)=1$, when $x\in A$, and 
$\bbone_{A}(x)=0$ otherwise. 
Then the family $\weight(s)=(\weight_k(s)\,|\,k\in K)\in\Zpl^K$ 
is called the weight of $s$ and satisfies
$\vecsum{\weight(s)}:=\sum_{k\in K}\weight_k(s)=n$. 

%%%%%%%%%%%%%%%%%%%%%%%%%%%%%%%%%%%%%%%%%%%%%%%%%%%%%%%%%%%%%%%%%%%%%%
 
\begin{lemma}\label{l487698}
Let $n,d\in\NN$, $(z_{j,k})\in\CC^{n\times d}$, and 
$x=(x_1,\dots,x_d)$, where $x_1,\dots,x_d$ are algebraically  
independent commuting indeterminates over $\CC$. For $m\in\Zpl^d$ 
with $\vecsum{m}=n$, $t\in\set{d}^n$ with $\weight(t)=m$, and
$Z'=(z_{j,t(r)})\in\CC^{n\times n}$, we then have 
\begin{align*}
\Coeff{x^m}{\prod_{j=1}^n\Bigl(\sum_{k=1}^dz_{j,k}x_k\Bigr)}
&=\sum_{\newatop{s\in\set{d}^n}{\weight(s)=m}}
\prod_{j=1}^n z_{j,s(j)}
=\frac{1}{m!}\per(Z').
\end{align*}
\end{lemma}
%%%%%%%%%%%%%%%%%%%%%%%%%%%%%%%%%%%%%%%%%%%%%%%%%%%%%%%%%%%%%%%%%%%%%%
Clearly, if $d\geq 2$ and $z_{j,d}=1$ for all $j\in\set{n}$,  then 
\begin{align*}
\Coeff{x^m}{\prod_{j=1}^n\Bigl(\sum_{k=1}^dz_{j,k}x_k\Bigr)}
=\Coeff{x_1^{m_1}\cdots x_{d-1}^{m_{d-1}}}{
\prod_{j=1}^n\Bigl(1+\sum_{k=1}^{d-1}z_{j,k}x_k\Bigr)}.
\end{align*}
Therefore, the coefficients considered in Lemma \ref{l487698} 
represent the natural generalization of the elementary symmetric 
polynomials, e.g.\ see \citet[Chapter 4, Section 2B]{MR1264417}.
These polynomials are members of the more general class of vector 
symmetric 
polynomials, which are also known under different names, such as 
multisymmetric polynomials or MacMahon symmetric functions, 
e.g.\ see \citet{MR2499933}. 

The next theorem follows from Lemma \ref{l487698} and 
Corollary \ref{c639860}.  
%%%%%%%%%%%%%%%%%%%%%%%%%%%%%%%%%%%%%%%%%%%%%%%%%%%%%%%%%%%%%%%%%%%%%%
\begin{theorem} \label{c873765}
Let $n,d\in\NN$, $(z_{j,k})\in\CC^{n\times d}$, 
$x=(x_1,\dots,x_d)$, where $x_1,\dots,x_d$ are 
algebraically independent commuting indeterminates over $\CC$.
For $m\in\Zpl^d$ with $\vecsum{m}=n$, we have 
\begin{align}\label{e76296}
\ABS{\Coeff{x^m}{\prod_{j=1}^n\Bigl(\sum_{k=1}^dz_{j,k}x_k\Bigr)}}
\leq \frac{n!}{m!}\prod_{k=1}^d
\Bigl(\frac{1}{\binomial{n}{m_k}}
\sum_{\newatop{I\subseteq{\set{n}}}{\card{I}=m_k}}
\prod_{j\in I}\abs{z_{j,k}}^2\Bigr)^{1/2}.
\end{align}
\end{theorem}
%%%%%%%%%%%%%%%%%%%%%%%%%%%%%%%%%%%%%%%%%%%%%%%%%%%%%%%%%%%%%%%%%%%%%%
\begin{remark}
\begin{enumerate}[(a)] 

\item 
In \eqref{e76296}, equality holds if one of the following conditions
holds:
\begin{itemize}\itemsep0pt

\item 
a number $k\in\set{d}$ exists such that $m_k\geq 1$ and
$\card{\{j\in\set{n}\,|\,z_{j,k}=0\}}\geq n-m_k+1$ or

\item 
there are numbers $\xi_j\in\CC$ with $\abs{\xi_j}=1$ for 
all $j\in\set{n}$ and $y_k\in\CC\setminus\{0\}$ 
for all $k\in\set{d}$ with $m_k\geq 1$, 
such that $z_{j,k}=\xi_{j}y_k$
for all such $j$ and $k$.

\end{itemize}

\item The right-hand side of \eqref{e76296} can be further estimated 
using Maclaurin's inequality, see Remark \ref{r6386}. 

\end{enumerate}
\end{remark}

%%%%%%%%%%%%%%%%%%%%%%%%%%%%%%%%%%%%%%%%%%%%%%%%%%%%%%%%%%%%%%%%%%%%%%
\section{Auxiliary inequalities}\label{s65275}
%%%%%%%%%%%%%%%%%%%%%%%%%%%%%%%%%%%%%%%%%%%%%%%%%%%%%%%%%%%%%%%%%%%%%%
The following proposition forms the main argument in the proof of 
Proposition~\ref{c845097} below, from which Theorem~\ref{c628560}
can be derived. Its proof and the proof of the subsequent remark can 
be found in Section~\ref{s376598}.
%%%%%%%%%%%%%%%%%%%%%%%%%%%%%%%%%%%%%%%%%%%%%%%%%%%%%%%%%%%%%%%%%%%%%%

\begin{proposition}\label{l845097}
Let $n\in\NN$, 
$\ell,m\in\Zpl$ with $m\leq \ell\leq n$, 
$g_1,\dots,g_n\in[0,\infty)$ and $g(I)=\prod_{j\in I}g_j$
for $I\subseteq\set{n}$. 
For $K\subseteq\set{n}$ with $\card{K}=\ell-m$, let
$h(K)\in[0,\infty)$. 
For $J\subseteq\set{n}$ with $\card{J}=\ell$, let 
$(g*_m h)(J)=\sum_{I\subseteq J:\,\card{I}=m}g(I)h(J\setminus I)$. 
Then we have
\begin{align}\label{e9843776aa}
\frac{1}{\binomial{n}{\ell}}
\sum_{\newatop{J\subseteq\set{n}}{\card{J}=\ell}}
\Bigl(\frac{1}{ \binomial{\ell}{m}}(g*_m h)(J)\Bigr)^2
&\leq
\Bigl(\frac{1}{\binomial{n}{m}}
\sum_{\newatop{I\subseteq \set{n}}{\card{I}=m}}g(I)^2\Bigr)
\Bigl(\frac{1}{\binomial{n}{\ell-m}}
\sum_{\newatop{J\subseteq\set{n}}{\card{J}=\ell-m}}h(J)^2\Bigr).
\end{align}
\end{proposition}

%%%%%%%%%%%%%%%%%%%%%%%%%%%%%%%%%%%%%%%%%%%%%%%%%%%%%%%%%%%%%%%%%%%%%%
We note that the stated assumption on $g$ seems to be somewhat 
restrictive. However, we do not know whether it can be dropped or not.
Further, there is a connection between the terms $(g*_m h)(J)$ used in 
Proposition~\ref{l845097} and the so-called subset convolution of 
two set functions. For two set functions $g$ and $h$ defined on the 
power set $2^{\set{n}}$ of $\set{n}$ with values in an arbitrary 
ring, the subset convolution $f:=g*h$ is a set function on 
$2^{\set{n}}$ defined by
$f(J)=\sum_{I\subseteq J}g(I)h(J\setminus I)$ for all 
$J\subseteq\set{n}$, e.g.\ see \citet{MR2402429}. 
If $g$ and $h$ are as in Proposition~\ref{l845097} and
$h$ is extended to a set function on $2^{\set{n}}$
such that $h(K)=0$ for all $K\subseteq\set{n}$ with 
$\card{K}\neq\ell-m$, then we can write $(g*_m h)(J)=(g*h)(J)$ for all
$J\subseteq\set{n}$ with $\card{J}=\ell$.

%%%%%%%%%%%%%%%%%%%%%%%%%%%%%%%%%%%%%%%%%%%%%%%%%%%%%%%%%%%%%%%%%%%%%%

\begin{remark} \label{r84876}
In \eqref{e9843776aa}, equality holds if and only if at least
one of the following five conditions is valid: 
\begin{enumerate}[(i)]

\item \label{r84876.i} 
$m\in\{0,\ell\}$ or

\item %\label{r84876.ii} 
$\card{\{j\in\set{n}\,|\,g_j>0\}}\leq m-1$  or 

\item %\label{r84876.iii}
$h(K)=0$ for all sets $K\subseteq\set{n}$ with 
$\card{K}=\ell-m$ or 

\item \label{r84876.iv}
$\ell=n$ and a number $x\in[0,\infty)$ exists such that 
$g(I)=x h(\set{n}\setminus I)$ for all $I\subseteq\set{n}$ 
with $\card{I}=m$ or 

\item \label{r84876.v}
$g_1=\dots=g_n$ and 
$h(K)=h(K')$ for all $K,K'\subseteq\set{n}$ with 
$\card{K}=\card{K'}=\ell-m$.

\end{enumerate}
\end{remark}
%%%%%%%%%%%%%%%%%%%%%%%%%%%%%%%%%%%%%%%%%%%%%%%%%%%%%%%%%%%%%%%%%%%%%%

The next proposition contains a generalization of (\ref{e37876}), see 
Remark~\ref{r8397698} below. 

%%%%%%%%%%%%%%%%%%%%%%%%%%%%%%%%%%%%%%%%%%%%%%%%%%%%%%%%%%%%%%%%%%%%%%
\begin{proposition}\label{c845097}
Let $n\in\NN$, $\emptyset\neq M\subseteq L\subseteq\set{n}$, 
$\ell=\card{L}$, $m=\card{M}\in\NN$, 
$Z=(z_{j,r})\in\CC^{\set{n}\times L}$, and
$g(I)=\prod_{j\in I}(\frac{1}{m}\sum_{r\in M}\abs{z_{j,r}}^2)$
for $I\subseteq\set{n}$. 
For $K\subseteq\set{n}$ with $\card{K}=\ell-m$, let 
$h(K)=\abs{\per(Z[K,L\setminus M])}$. 
Set $C=C(\ell,m,n)
=\frac{\binomial{\ell}{m}\binomial{n-\ell+m}{m}}{\binomial{n}{m}}$. 
Then 
\begin{align}\label{e9843776}
\sum_{\newatop{J\subseteq\set{n}}{\card{J}=\ell}}\Abs{\per(Z[J,L])}^2
&\leq (m!)^2C
\Bigl(\sum_{\newatop{I\subseteq \set{n}}{\card{I}=m}}g(I)\Bigr)
\sum_{\newatop{J\subseteq\set{n}}{\card{J}=\ell-m}}h(J)^2.
\end{align}
\end{proposition}
%%%%%%%%%%%%%%%%%%%%%%%%%%%%%%%%%%%%%%%%%%%%%%%%%%%%%%%%%%%%%%%%%%%%%%
\begin{proof} Let $J\subseteq\set{n}$ with $\card{J}=\ell$. 
The Laplace expansion for permanents gives 
\begin{align}
\per(Z[J,L])
&=\sum_{\newatop{I\subseteq J}{\card{I}=m}}
\per(Z[I,M])\,\per(Z[J\setminus I,L\setminus M]),\label{e7386509}
\end{align}
see \citet[Theorem 1.2, page 16]{MR504978}.
From \eqref{e75870}, it follows that, for $I\subseteq J$ 
with $\card{I}=m$, 
\begin{align}\label{e873296}
\Abs{\per(Z[I,M])}
\leq m!\prod_{j\in I}\Bigl(\frac{1}{m}\sum_{r\in M}\abs{z_{j,r}}^2
\Bigr)^{1/2}
=m!\sqrt{g(I)}.
\end{align}
From \eqref{e7386509} and \eqref{e873296}, we get
\begin{align*}
\sum_{\newatop{J\subseteq\set{n}}{\card{J}=\ell}}
\Abs{\per(Z[J,L])}^2
&\leq \sum_{\newatop{J\subseteq\set{n}}{\card{J}=\ell}}
\Bigl(\sum_{\newatop{I\subseteq J}{\card{I}=m}}
\Abs{\per(Z[I,M])}\,h(J\setminus I)\Bigr)^2 \\
&\leq(m!)^2\sum_{\newatop{J\subseteq\set{n}}{\card{J}=\ell}}
\Bigl(\sum_{\newatop{I\subseteq J}{\card{I}=m}}
\sqrt{g(I)}\,h(J\setminus I)\Bigr)^2. 
\end{align*}
Proposition \ref{l845097} together with
\begin{align}\label{e31855}
C=\frac{\binomial{n}{\ell}\binomial{\ell}{m}^2}{
\binomial{n}{m}\binomial{n}{\ell-m}}
\end{align}
implies \eqref{e9843776}.
\end{proof}

%%%%%%%%%%%%%%%%%%%%%%%%%%%%%%%%%%%%%%%%%%%%%%%%%%%%%%%%%%%%%%%%%%%%%%
\begin{remark}\label{r8397698}
Let the assumptions of Proposition \ref{c845097} be valid. 
\begin{enumerate}[(a)]

\item Let $s\in M$. 
If $\abs{z_{j,r}}=\abs{z_{j,s}}$ for all $j\in\set{n}$ and $r\in M$, 
then 
\eqref{e9843776} implies that
\begin{align}\label{e732860}
\sum_{\newatop{J\subseteq\set{n}}{\card{J}=\ell}}\Abs{\per(Z[J,L])}^2
&\leq (m!)^2C
\Bigl(\sum_{\newatop{I\subseteq \set{n}}{\card{I}=m}}
\prod_{j\in I}\abs{z_{j,s}}^2\Bigr)
\sum_{\newatop{J\subseteq\set{n}}{\card{J}=\ell-m}}h(J)^2.
\end{align}
In particular, if $M=\{s\}$, i.e.\ $m=1$, then 
\eqref{e732860} reduces to \eqref{e37876}.

\item If $m=\ell$, i.e.\ $M=L$, then \eqref{e9843776} gives 
\begin{align} \label{e718476}
\sum_{\newatop{J\subseteq\set{n}}{\card{J}=\ell}}\Abs{\per(Z[J,L])}^2
\leq(\ell!)^2 
\Bigl(\sum_{\newatop{I\subseteq \set{n}}{\card{I}=\ell}}g(I)\Bigr),
\end{align}
which also easily follows from \eqref{e75870}. 
In the case $m=\ell=n$, i.e.\ $M=L=\set{n}$, \eqref{e718476}
reduces to \eqref{e75870} applied to the transpose of $Z$. 
\end{enumerate}
\end{remark}
%%%%%%%%%%%%%%%%%%%%%%%%%%%%%%%%%%%%%%%%%%%%%%%%%%%%%%%%%%%%%%%%%%%%%%

The next result is a generalization of Theorem~\ref{t762986}. 
%%%%%%%%%%%%%%%%%%%%%%%%%%%%%%%%%%%%%%%%%%%%%%%%%%%%%%%%%%%%%%%%%%%%%%

\begin{theorem}\label{c628560}
Let $n\in\NN$, $\emptyset\neq L\subseteq\set{n}$, $\ell=\card{L}$, 
$d\in\set{\ell}$, 
$M_1,\dots,M_d\subseteq L$ be pairwise disjoint sets with 
$\bigcup_{k=1}^dM_k=L$ and $m_k=\card{M_k}\in\NN$ for $k\in\set{d}$.
Let $Z=(z_{j,r})\in\CC^{\set{n}\times L}$ 
and $\alpha_{j,k}=\frac{1}{m_k}\sum_{r\in M_k}\abs{z_{j,r}}^2$
for $j\in\set{n}$ and $k\in\set{d}$. Then 
\begin{align}\label{e84376}
\sum_{\newatop{J\subseteq\set{n}}{\card{J}=\ell}}
\Abs{\per(Z[J,L])}^2
&\leq (\ell!)^2\binomial{n}{\ell}
\prod_{k=1}^d\Bigl(\frac{1}{\binomial{n}{m_k}}
\sum_{\newatop{I\subseteq \set{n}}{\card{I}=m_k}}
\prod_{j\in I}\alpha_{j,k}\Bigr).
\end{align}
\end{theorem}
%%%%%%%%%%%%%%%%%%%%%%%%%%%%%%%%%%%%%%%%%%%%%%%%%%%%%%%%%%%%%%%%%%%%%%
\begin{proof}
We use induction over $d$. For $d=1$, i.e.\ $M_1=L$ and $m_1=\ell$, 
\eqref{e84376} immediately follows from \eqref{e718476}.
In the proof of the assertion for general 
$d\in\set{\ell}\setminus\{1\}$ for $\ell\geq 2$,  
we assume its validity for $d-1$. Then \eqref{e9843776} gives 
\begin{align*}
\sum_{\newatop{J\subseteq\set{n}}{\card{J}=\ell}}
\Abs{\per(Z[J,L])}^2
&\leq (m_d!)^2C(\ell,m_d,n)
\Bigl(\sum_{\newatop{I\subseteq \set{n}}{\card{I}=m_d}}
\prod_{j\in I}\alpha_{j,d}\Bigr)
\sum_{\newatop{J\subseteq\set{n}}{\card{J}=\ell-m_d}}
\Abs{\per(Z[J,L\setminus M_d])}^2\\
&\leq C'(\ell,m_d,n)
\prod_{k=1}^{d}\Bigl(\frac{1}{\binomial{n}{m_k}}
\sum_{\newatop{I\subseteq \set{n}}{\card{I}=m_k}}
\prod_{j\in I}\alpha_{j,k}\Bigr),
\end{align*}
where, by using \eqref{e31855},
\begin{align*}
C'(\ell,m_d,n)
&=(m_d!)^2C(\ell,m_d,n)((\ell-m_d)!)^2\binomial{n}{\ell-m_d}
\binomial{n}{m_d}
=(\ell!)^2\binomial{n}{\ell}.
\end{align*}
This completes the proof of \eqref{e84376}.  
\end{proof}
%%%%%%%%%%%%%%%%%%%%%%%%%%%%%%%%%%%%%%%%%%%%%%%%%%%%%%%%%%%%%%%%%%%%%%
\begin{remark} 
\begin{enumerate}[(a)]

\item If $\ell=\card{L}=n$, i.e.\ $L=\set{n}$, then 
Theorem~\ref{c628560} reduces to  Theorem \ref{t762986}. 
If $d=1$, $M_1=L$, and $m_1=\ell$, 
then \eqref{e84376} easily follows from \eqref{e75870}. 
On the other hand, if $d=\ell$ and 
$m_1=\dots=m_\ell=1$, then \eqref{e84376} reduces to~\eqref{e738659}. 

\item 
In the proof of Proposition \ref{c845097}, the Hadamard type 
inequality \eqref{e75870} has been used, see \eqref{e873296}. 
However, in the case $m=1$, \eqref{e873296} is trivially valid.
Therefore, in the case $m_1=\dots=m_d=1$, Theorem~\ref{c628560} 
is proved without using \eqref{e75870}. 
Furthermore, if $d=\ell=n$ and 
$m_1=\dots=m_n=1$, Theorem~\ref{c628560} reduces to \eqref{e75870}. 
In this respect, the present paper is self-contained. 

\item The proof of Proposition \ref{c845097} is heavily based on 
\eqref{e873296}. However, instead we could use other inequalities 
of the form
\begin{align} \label{e7825876}
\Abs{\per(Z[I,M])}
\leq m!\sqrt{\widetilde{g}(I)},
\end{align}
for $I\subseteq J\subseteq\set{n}$ with $\card{J}=\ell$ and 
$\card{I}=m$, where $\widetilde{g}(I)=\prod_{j\in I}\widetilde{g}_j$
and $\widetilde{g}_1,\dots,\widetilde{g}_n\in[0,\infty)$. 
For instance, in the case $Z\in\{0,1\}^{\set{n}\times L}$, 
\eqref{e7825876} is valid with
\begin{align*}
\widetilde{g}_j
=\Bigl(\frac{\eta(\lambda_j)}{(m!)^{1/m}}\Bigr)^2,
\end{align*}
where $j\in\set{n}$, $\lambda_j=\sum_{r\in M} z_{j,r}$, 
$\eta(k)=(k!)^{1/k}$ for  $k\in\NN$, and $\eta(0)=0$.
This is a consequence of the Br\'{e}gman-Minc permanent inequality, 
which says that, for $Z\in\{0,1\}^{n\times n}$, 
\begin{align*} 
\per(Z)\leq \prod_{j=1}^n\eta\Bigl(\sum_{r=1}^nz_{j,r}\Bigr),
\end{align*}
cf.\ \cite{MR0327788} and \cite{MR0155843}. 
Using an inequality like \eqref{e7825876}
instead of \eqref{e873296}, one would be able to show 
new inequalities similar to those given in 
Proposition~\ref{c845097} and Theorems~\ref{c628560},~\ref{t762986}. 
However, we do not follow this idea here.
\end{enumerate}
\end{remark}  

%%%%%%%%%%%%%%%%%%%%%%%%%%%%%%%%%%%%%%%%%%%%%%%%%%%%%%%%%%%%%%%%%%%%%%
\section{Proof of Proposition \ref{l845097} and 
Remark \ref{r84876}}\label{s376598}
%%%%%%%%%%%%%%%%%%%%%%%%%%%%%%%%%%%%%%%%%%%%%%%%%%%%%%%%%%%%%%%%%%%%%%
The proof of Proposition \ref{l845097} is based on a generalization 
of the approach used by \citet[Section 3]{MR2275869} in 
the second proof of their Theorem~1.1. 
Because of our general assumptions, our proof
is somewhat technical. We need the following two lemmata.
As usual, for $x,y\in\RR$, let $x\wedge y=\min\{x,y\}$
and $x\vee y=\max\{x,y\}$. 
%%%%%%%%%%%%%%%%%%%%%%%%%%%%%%%%%%%%%%%%%%%%%%%%%%%%%%%%%%%%%%%%%%%%%%
\begin{lemma}\label{l47687}
If $x,y\in\CC$ and $m,n\in\Zpl$, then
\begin{align*} 
\sum_{k=0}^{m\wedge n}\binomial{x}{m-k}
\binomial{y}{n-k}\binomial{x+y+k}{k}
=\binomial{x+n}{m}\binomial{y+m}{n}.
\end{align*}
\end{lemma}
%%%%%%%%%%%%%%%%%%%%%%%%%%%%%%%%%%%%%%%%%%%%%%%%%%%%%%%%%%%%%%%%%%%%%%
\begin{proof} 
This is an immediate consequence of the Pfaff-Saalsch\"utz 
identity from the theory of hypergeometric series, e.g.\ see 
\citet[Formula (1)]{MR993660} or \citet[Formula (19)]{MR1313292}. 
Another independent and short proof can be found 
in \citet[proof of Formula (1)]{MR773560}. 
\end{proof}
%%%%%%%%%%%%%%%%%%%%%%%%%%%%%%%%%%%%%%%%%%%%%%%%%%%%%%%%%%%%%%%%%%%%%%
The previous lemma can easily be used to prove the next result.
We note that in the first attempt to prove Proposition \ref{l845097},
one task was to find non-negative numbers $f(a,b)$ satisfying 
\eqref{e87396} and \eqref{e7255}, which was not that easy. 
%%%%%%%%%%%%%%%%%%%%%%%%%%%%%%%%%%%%%%%%%%%%%%%%%%%%%%%%%%%%%%%%%%%%%%

\begin{lemma}\label{l47897}
Let $n\in\NN$, $\ell,m\in\Zpl$ with $m\leq \ell\leq n$. Let 
\begin{align*}
C=C(\ell ,m,n)
=\frac{\binomial{\ell }{m}\binomial{n-\ell +m}{m}}{
\binomial{n}{m}}, \qquad
f(a,b)=f_{\ell,m,n}(a,b)
=\frac{\binomial{n-\ell }{m-a-b}\binomial{\ell }{b}}{
\binomial{m-a}{b}^2 \binomial{n}{m-a}}
\end{align*}
for $(a,b)\in \Zpl^2$ with $a+b\leq m$. 
Then we always have $f(a,b)\in[0,\infty)$ and,
for all $a\in\{0,\dots,m\}$, 
\begin{align}\label{e87396}
\sum_{b=0}^{m-a}f(a,b)\binomial{m-a}{b}^2=1.
\end{align}
Further, for all $b\in\Zpl$ with $b\leq m\wedge(\ell -m)$, 
\begin{align}\label{e7255}
\sum_{a=0\vee(2m-\ell)}^{m-b} f(a,b)\binomial{m-b}{a}
\binomial{\ell -m-b}{m-a-b}\binomial{n-\ell +b}{b}
=C.
\end{align}
\end{lemma}
%%%%%%%%%%%%%%%%%%%%%%%%%%%%%%%%%%%%%%%%%%%%%%%%%%%%%%%%%%%%%%%%%%%%%%
\begin{proof}
It is clear that always $f(a,b)\in[0,\infty)$. 
Using Vandermonde's identity for binomial coefficients, we obtain
\begin{align*}
\sum_{b=0}^{m-a}f(a,b)\binomial{m-a}{b}^2
=\frac{1}{\binomial{n}{m-a}}
\sum_{b=0}^{m-a}\binomial{n-\ell }{m-a-b}\binomial{\ell }{b}=1. 
\end{align*}
Further, for $a,b\in\Zpl$ with $b\leq m\wedge(\ell -m)$
and $0\vee(2m-\ell )\leq a\leq m-b$, it is easily shown that
\begin{align*}
f(a,b)\binomial{m-b}{a}
\binomial{\ell -m-b}{m-a-b}\binomial{n-\ell +b}{b}
=C\frac{\binomial{n-\ell }{m-a-b}\binomial{\ell -m}{m-a}
\binomial{n-m+a}{a}}{\binomial{n-\ell +m}{m-b}\binomial{\ell -b}{m}}.
\end{align*}
This together with Lemma \ref{l47687} implies \eqref{e7255}. 
\end{proof}
%%%%%%%%%%%%%%%%%%%%%%%%%%%%%%%%%%%%%%%%%%%%%%%%%%%%%%%%%%%%%%%%%%%%%%
\begin{remark} \label{r4186}
In Lemma~\ref{l47897}, we have $f(a,b)>0$ if and only if 
$m-a-b\leq n-\ell$.
\end{remark}
%%%%%%%%%%%%%%%%%%%%%%%%%%%%%%%%%%%%%%%%%%%%%%%%%%%%%%%%%%%%%%%%%%%%%%
\begin{proof}[Proof of Proposition \ref{l845097}]
In \eqref{e9843776aa}, equality holds for $m=0$. 
In what follows, let us assume that $1\leq m\leq \ell\leq n$. 
Let $J\subseteq\set{n}$ with $\card{J}=\ell$. 
Then
\begin{align}
((g*_m h)(J))^2
&=\Bigl(\sum_{\newatop{I\subseteq J}{\card{I}=m}}
g(I)h(J\setminus I)\Bigr)^2\nonumber\\
&=\sum_{\newatop{I_1\subseteq J}{\card{I_1}=m}}
\sum_{\newatop{I_2\subseteq J}{\card{I_2}=m}}
g(I_1)g(I_2)h(J\setminus I_1)h(J\setminus I_2)\nonumber\\
&= \sum_{a=0\vee(2m-\ell)}^m
\sum_{\newatop{D\subseteq J}{\card{D}=a}}
\sum_{\newatop{J_1\subseteq J\setminus D}{\card{J_1}=m-a}}
\sum_{\newatop{J_2\subseteq J\setminus(D\cup J_1)}{
\card{J_2}=m-a}} g(D)^2\nonumber\\
&\quad{}\times g(J_1)g(J_2)
h(J\setminus (D\cup J_1))h(J\setminus (D\cup J_2)),\label{e7215286}
\end{align}
where, in \eqref{e7215286}, we changed variables, namely 
$I_1=D\cup J_1$, $I_2=D\cup J_2$, with $D\subseteq J$, 
$\card{D}=a$, $J_1,J_2\subseteq J\setminus D$, and
$J_1\cap J_2=\emptyset$. 
Here, $\card{D\cup J_1\cup J_2}
=\card{I_1\cup (I_2\setminus D)}=2m-a\leq \ell$, 
such that $a\in\{0\vee(2m-\ell),\dots,m\}$. 
For $(a,b)\in\Zpl^2$ with $a+b\leq m$, 
let $f(a,b)=f_{\ell,m,n}(a,b)$ be defined as in Lemma \ref{l47897}. 
For $J_1,J_2\subseteq J$ with $\card{J_1}=m-a=\card{J_2}$, we have 
\begin{align*}
\sum_{b=0}^{m-a}f(a,b)
\sum_{\newatop{I_1\subseteq J_1}{\card{I_1}=m-a-b}}
\sum_{\newatop{I_2\subseteq J_2}{\card{I_2}=b}}1
=\sum_{b=0}^{m-a}f(a,b)\binomial{m-a}{b}^2=1,
\end{align*}
and therefore 
\begin{align*}
((g*_m h)(J))^2
&=\sum_{a=0\vee(2m-\ell)}^m
\sum_{\newatop{D\subseteq J}{\card{D}=a}}
\sum_{\newatop{J_1\subseteq J\setminus D}{\card{J_1}=m-a}}
\sum_{\newatop{J_2\subseteq J\setminus(D\cup J_1)}{
\card{J_2}=m-a}}\sum_{b=0}^{m-a}f(a,b)\\
&\quad{}\times
\sum_{\newatop{I_1\subseteq J_1}{\card{I_1}=m-a-b}}
\sum_{\newatop{I_2\subseteq J_2}{\card{I_2}=b}}
g(D\cup I_1)g(I_2)h(J\setminus (D\cup J_1))\\
&\quad{}\times g(J_1\setminus I_1)g(D\cup (J_2\setminus I_2))
h(J\setminus (D\cup J_2)). 
\end{align*}
Now we use the inequality $xy\leq \frac{1}{2}(x^2+y^2)$ for 
$x,y\in[0,\infty)$ and obtain
\begin{align}
((g*_m h)(J))^2
&\leq \frac{1}{2}\sum_{b=0}^{m\wedge(\ell-m)}
\sum_{a=0\vee(2m-\ell)}^{m-b}f(a,b)
\sum_{\newatop{D\subseteq J}{\card{D}=a}}
\sum_{\newatop{J_1\subseteq J\setminus D}{\card{J_1}=m-a}}
\sum_{\newatop{J_2\subseteq J\setminus(D\cup J_1)}{
\card{J_2}=m-a}}\nonumber\\
&\quad{}\times
\sum_{\newatop{I_1\subseteq J_1}{\card{I_1}=m-a-b}}
\sum_{\newatop{I_2\subseteq J_2}{\card{I_2}=b}}
\Bigl(\bigl(g(D\cup I_1)g(I_2)h(J\setminus (D\cup J_1))
\bigr)^2\nonumber\\
&\qquad\qquad{}+
\bigl(g(J_1\setminus I_1)g(D\cup(J_2\setminus I_2))
h(J\setminus (D\cup J_2))\bigr)^2\Bigr).\label{e7629608} 
\end{align}
This gives $((g*_m h)(J))^2\leq\frac{1}{2}(T_1(J)+T_2(J))$, where 
\begin{align*}
T_1(J)&=\sum_{b=0}^{m\wedge(\ell-m)}
\sum_{a=0\vee(2m-\ell)}^{m-b}f(a,b)
\sum_{\newatop{D\subseteq J}{\card{D}=a}}
\sum_{\newatop{J_1\subseteq J\setminus D}{\card{J_1}=m-a}}
\sum_{\newatop{J_2\subseteq J\setminus(D\cup J_1)}{
\card{J_2}=m-a}}
\sum_{\newatop{I_1\subseteq J_1}{\card{I_1}=m-a-b}}
\sum_{\newatop{I_2\subseteq J_2}{\card{I_2}=b}}\\
&\quad{}\times 
\bigl(g(D\cup I_1)g(I_2)h(J\setminus (D\cup J_1))\bigr)^2,\\
T_2(J)&=\sum_{b=0}^{m\wedge(\ell-m)}
\sum_{a=0\vee(2m-\ell)}^{m-b}f(a,b)
\sum_{\newatop{D\subseteq J}{\card{D}=a}}
\sum_{\newatop{J_1\subseteq J\setminus D}{\card{J_1}=m-a}}
\sum_{\newatop{J_2\subseteq J\setminus(D\cup J_1)}{
\card{J_2}=m-a}}
\sum_{\newatop{I_1\subseteq J_1}{\card{I_1}=m-a-b}}
\sum_{\newatop{I_2\subseteq J_2}{\card{I_2}=b}}\\
&\quad{}\times
\bigl(g(J_1\setminus I_1)g(D\cup(J_2\setminus I_2))
h(J\setminus (D\cup J_2))\bigr)^2.
\end{align*}
By symmetry, we have $T_1(J)=T_2(J)$. In fact, 
\begin{align}
T_1(J)&=\sum_{b=0}^{m\wedge(\ell-m)}
\sum_{a=0\vee(2m-\ell)}^{m-b}f(a,b)
\sum_{\newatop{D\subseteq J}{\card{D}=a}}
\sum_{\newatop{J_1\subseteq J\setminus D}{\card{J_1}=m-a}}
\sum_{\newatop{J_2\subseteq J\setminus(D\cup J_1)}{
\card{J_2}=m-a}}
\sum_{\newatop{I_2\subseteq J_2}{\card{I_2}=m-a-b}}
\sum_{\newatop{I_1\subseteq J_1}{\card{I_1}=b}}\nonumber\\
&\quad{}\times \bigl(g(D\cup I_2)g(I_1)h(J\setminus (D\cup J_2))
\bigr)^2\label{e732770}\\
&=\sum_{b=0}^{m\wedge(\ell-m)}
\sum_{a=0\vee(2m-\ell)}^{m-b}f(a,b)
\sum_{\newatop{D\subseteq J}{\card{D}=a}}
\sum_{\newatop{J_1\subseteq J\setminus D}{\card{J_1}=m-a}}
\sum_{\newatop{J_2\subseteq J\setminus(D\cup J_1)}{
\card{J_2}=m-a}}
\sum_{\newatop{I_2'\subseteq J_2}{\card{I_2'}=b}}
\sum_{\newatop{I_1'\subseteq J_1}{\card{I_1'}=m-a-b}}\nonumber\\
&\quad{}\times
\bigl(g( J_1\setminus I_1')g(D\cup(J_2\setminus I_2'))
h(J\setminus (D\cup J_2))\bigr)^2\label{e732771}\\
&=T_2(J).\nonumber
\end{align}
Here, \eqref{e732770} follows by interchanging $J_1$ with $J_2$ 
and $I_1$ with $I_2$; further, in \eqref{e732771},
we changed variables, that is, $I_1=J_1\setminus I_1'$ and
$I_2=J_2\setminus I_2'$. Consequently, 
\begin{align}\label{e635581}
\sum_{\newatop{J\subseteq\set{n}}{\card{J}=\ell}}((g*_m h)(J))^2
&\leq T_0
:=\sum_{\newatop{J\subseteq\set{n}}{\card{J}=\ell}}T_1(J)
=\sum_{b=0}^{m\wedge(\ell-m)}
\sum_{a=0\vee(2m-\ell)}^{m-b}f(a,b)T_3(a,b),
\end{align}
where, for $a,b\in\Zpl$ with $b\leq m\wedge(\ell-m)$
and $2m-\ell\leq a\leq m-b$, 
\begin{align}
T_3(a,b)
&=\sum_{\newatop{J\subseteq\set{n}}{\card{J}=\ell}}
\sum_{\newatop{D\subseteq J}{\card{D}=a}}
\sum_{\newatop{J_1\subseteq J\setminus D}{\card{J_1}=m-a}}
\sum_{\newatop{J_2\subseteq J\setminus(D\cup J_1)}{
\card{J_2}=m-a}}
\sum_{\newatop{I_1\subseteq J_1}{\card{I_1}=m-a-b}}
\sum_{\newatop{I_2\subseteq J_2}{\card{I_2}=b}}\nonumber\\
&\quad{}\times
\bigl(g(D\cup I_1)g(I_2)h(J\setminus (D\cup J_1))
\bigr)^2\nonumber\\
&=\sum_{\newatop{J\subseteq\set{n}}{\card{J}=\ell}}
\sum_{\newatop{G\subseteq J}{\card{G}=m}}
\sum_{\newatop{D\subseteq G}{\card{D}=a}}
\sum_{\newatop{J_2\subseteq J\setminus G}{\card{J_2}=m-a}}
\sum_{\newatop{I_1\subseteq G\setminus D}{\card{I_1}=m-a-b}}
\sum_{\newatop{I_2\subseteq J_2}{\card{I_2}=b}}
\bigl(g(D\cup I_1)g(I_2)h(J\setminus G)\bigr)^2\label{eq71765}\\
&=\sum_{\newatop{G\subseteq \set{n}}{\card{G}=m}}
\sum_{\newatop{J'\subseteq\set{n}\setminus G}{\card{J'}=\ell-m}}
\sum_{\newatop{D\subseteq G}{\card{D}=a}}
\sum_{\newatop{I_1\subseteq G\setminus D}{\card{I_1}=m-a-b}}
\sum_{\newatop{J_2\subseteq J'}{\card{J_2}=m-a}}
\sum_{\newatop{I_2\subseteq J_2}{\card{I_2}=b}}
\bigl(g(D\cup I_1)g(I_2)h(J')\bigr)^2.\label{eq71766}
\end{align}
In \eqref{eq71765} and \eqref{eq71766}, we changed variables, i.e.\ 
$J_1=G\setminus D$ and $J=J'\cup G$, respectively. 
Now, letting $I_1=I_1'\setminus D$ and $J_2=I_2\cup J_2'$, we get
\begin{align}
T_3(a,b)
&=\sum_{\newatop{G\subseteq \set{n}}{\card{G}=m}}
\sum_{\newatop{J'\subseteq\set{n}\setminus G}{\card{J'}=\ell-m}}
\sum_{\newatop{I_1'\subseteq G}{\card{I_1'}=m-b}}
\sum_{\newatop{D\subseteq I_1'}{\card{D}=a}}
\sum_{\newatop{I_2\subseteq J'}{\card{I_2}=b}}
\sum_{\newatop{J_2'\subseteq J'\setminus I_2}{\card{J_2'}=m-a-b}}
\bigl(g(I_1')g(I_2)h(J')\bigr)^2\nonumber\\
&=\sum_{\newatop{G\subseteq \set{n}}{\card{G}=m}}
\sum_{\newatop{J\subseteq\set{n}\setminus G}{\card{J}=\ell-m}}
\sum_{\newatop{I_1\subseteq G}{\card{I_1}=m-b}}
\sum_{\newatop{D\subseteq I_1}{\card{D}=a}}
\sum_{\newatop{I_2\subseteq J}{\card{I_2}=b}}
\sum_{\newatop{J_2\subseteq J\setminus I_2}{\card{J_2}=m-a-b}}
\bigl(g(I_1)g(I_2)h(J)\bigr)^2\nonumber\\
&=\binomial{m-b}{a}\binomial{\ell-m-b}{m-a-b}T_4(b),
\label{e75971}
\end{align}
where
\begin{align*}
T_4(b)
&=\sum_{\newatop{G\subseteq \set{n}}{\card{G}=m}}
\sum_{\newatop{J\subseteq\set{n}\setminus G}{\card{J}=\ell-m}}
\sum_{\newatop{I_1\subseteq G}{\card{I_1}=m-b}}
\sum_{\newatop{I_2\subseteq J}{\card{I_2}=b}}
\bigl(g(I_1)g(I_2)h(J)\bigr)^2.
\end{align*}
Interchanging the first two sums
and changing variables such that $G=I_1\cup G'$, we obtain 
\begin{align}
T_4(b)
&=\sum_{\newatop{J\subseteq\set{n}}{\card{J}=\ell-m}}
\sum_{\newatop{G\subseteq \set{n}\setminus J}{\card{G}=m}}
\sum_{\newatop{I_1\subseteq G}{\card{I_1}=m-b}}
\sum_{\newatop{I_2\subseteq J}{\card{I_2}=b}}
\bigl(g(I_1)g(I_2)h(J)\bigr)^2\nonumber\\
&=\sum_{\newatop{J\subseteq\set{n}}{\card{J}=\ell-m}}
\sum_{\newatop{I_1\subseteq \set{n}\setminus J}{\card{I_1}=m-b}}
\sum_{\newatop{G'\subseteq \set{n}\setminus (J\cup I_1)}{\card{G'}=b}}
\sum_{\newatop{I_2\subseteq J}{\card{I_2}=b}}
\bigl(g(I_1)g(I_2)h(J)\bigr)^2 
=\binomial{n-\ell+b}{b}T_5(b)\label{e648601}
\end{align}
with 
\begin{align*}
T_5(b)
&=\sum_{\newatop{J\subseteq\set{n}}{\card{J}=\ell-m}}
\sum_{\newatop{I_1\subseteq \set{n}\setminus J}{\card{I_1}=m-b}}
\sum_{\newatop{I_2\subseteq J}{\card{I_2}=b}}
\bigl(g(I_1)g(I_2)h(J)\bigr)^2.
\end{align*}
Using \eqref{e635581}, \eqref{e75971}, \eqref{e648601},
and \eqref{e7255}, we get
\begin{align}
T_0
&=\sum_{b=0}^{m\wedge(\ell-m)}
\sum_{a=0\vee(2m-\ell)}^{m-b}f(a,b)T_3(a,b)\nonumber\\
&=\sum_{b=0}^{m\wedge(\ell-m)} 
\sum_{a=0\vee(2m-\ell)}^{m-b}f(a,b)
\binomial{m-b}{a}\binomial{\ell-m-b}{m-a-b}T_4(b)\nonumber\\
&=C\sum_{b=0}^{m\wedge(\ell-m)}T_5(b) 
=C\sum_{\newatop{J\subseteq\set{n}}{\card{J}=\ell-m}}
T_6(J)h(J)^2,\label{e652551}
\end{align}
where $C=C(\ell,m,n)
=\frac{\binomial{\ell}{m}\binomial{n-\ell+m}{m}}{\binomial{n}{m}}$
and, for $J\subseteq\set{n}$ with $\card{J}=\ell-m$, 
\begin{align}
T_6(J)
&=\sum_{b=0}^{m\wedge(\ell-m)}
\sum_{\newatop{I_1\subseteq \set{n}\setminus J}{\card{I_1}=m-b}}
\sum_{\newatop{I_2\subseteq J}{\card{I_2}=b}}
\bigl(g(I_1)g(I_2)\bigr)^2\nonumber\\
&=\sum_{b=0}^{m\wedge(\ell-m)}
\sum_{\newatop{I\subseteq \set{n}}{\card{I\cap J}=b,
\card{I\setminus J}=m-b}}
\bigl(g(I\setminus J)g(I\cap J)\bigr)^2\label{e61921743}\\
&=\sum_{\newatop{I\subseteq \set{n}}{\card{I}=m}}
\bigl(g(I\setminus J)g(I\cap J)\bigr)^2
=\sum_{\newatop{I\subseteq \set{n}}{\card{I}=m}}g(I)^2. 
\label{e76254761}
\end{align}
In \eqref{e61921743}, we changed variables according to 
$I_1=I\setminus J$ and $I_2=I\cap J$. 
Inequality \eqref{e9843776aa} now follows from \eqref{e635581},
\eqref{e652551}, \eqref{e31855}, and \eqref{e76254761}. 
\end{proof}
%%%%%%%%%%%%%%%%%%%%%%%%%%%%%%%%%%%%%%%%%%%%%%%%%%%%%%%%%%%%%%%%%%%%%%

\begin{proof}[Proof of Remark \ref{r84876} (Sufficiency)]
It is easy to verify that, if at least one of the  conditions 
(\ref{r84876.i})--(\ref{r84876.v}) is valid, then, 
in \eqref{e9843776aa}, equality holds.
Further, for $\ell=n$, it is clear that equality in \eqref{e9843776aa}
is equivalent to the existence of a number $x\in[0,\infty)$ such that 
$g(I)=x h(\set{n}\setminus I)$ for all $I\subseteq\set{n}$ 
with $\card{I}=m$. 
\end{proof}
%%%%%%%%%%%%%%%%%%%%%%%%%%%%%%%%%%%%%%%%%%%%%%%%%%%%%%%%%%%%%%%%%%%%%%
\begin{proof}[Proof of Remark \ref{r84876} (Necessity)]
Let us assume that, in \eqref{e9843776aa}, equality holds, where 
conditions (\ref{r84876.i})--(\ref{r84876.iv}) do not hold. 
So, let us assume that $1\leq m<\ell<n$ and sets 
$I_0\subseteq\set{n}$ and $K_0\subseteq\set{n}$ exist such that 
$\card{I_0}=m$, $\prod_{j\in I_0}g_j=g(I_0)>0$,
$\card{K_0}=\ell-m$, and $h(K_0)>0$.  
Our goal is to show that condition (\ref{r84876.v}) is valid, i.e.\ 
$g_1=\dots=g_n$ and $h(K)=h(K')$ for all 
$K,K'\subseteq\set{n}$ with $\card{K}=\card{K'}=\ell-m$. 
\begin{enumerate}[(a)]

\item 
Let us now explain the main argument used here. 
In the proof of Proposition~\ref{l845097} (see~\eqref{e7629608}), 
it was used that 
$xy\leq\frac{1}{2}(x^2+y^2)$ for $x,y\in[0,\infty)$, where 
equality holds if and only if $x=y$. Since no other 
inequalities were used, it follows that 
\begin{align}\label{e1543822}
g(D)g(I_1)g(I_2)h(J\setminus (D\cup J_1))
=g(D)g(J_1\setminus I_1)g(J_2\setminus I_2)h(J\setminus (D\cup J_2)), 
\end{align}
whenever 
$J\subseteq\set{n}$ with $\card{J}=\ell$,
$D\subseteq J$,
$I_1\subseteq J_1\subseteq J\setminus D$,
and $I_2\subseteq J_2\subseteq J\setminus (D\cup J_1)$
with $\card{J_1}=\card{J_2}=m-\card{D}$ and 
$\card{I_1}=\card{J_1}-\card{I_2}\leq n-\ell$.
In this case, we say that $(J,D,J_1,I_1,J_2,I_2)$ is an 
admissible family. In particular, $D$, $J_1$, and $J_2$ are pairwise 
disjoint sets with $D\cup J_1\cup J_2\subseteq J$. 
Let $a=\card{D}$ and $b=\card{I_2}$. Then 
$0\leq b\leq m\wedge(\ell-m)$ and 
$0\vee(2m-\ell)\vee (m-n+\ell-b)\leq a\leq m-b$, 
$\card{J_1}=\card{J_2}=m-a$ and
$\card{I_1}=m-a-b$. 
It is noteworthy that, in the application of \eqref{e7629608},
the values $f(a,b)$ must be strictly positive, so that 
we have to assume the inequality 
$\card{J_1}-\card{I_2}=m-a-b\leq n-\ell$, see Remark~\ref{r4186}.

\item  \label{r84876.b}
In what follows, we construct two special admissible families.
Let $I_2=K_0\cap I_0$ and $b=\card{I_2}$. 
Clearly, we have $0\leq b\leq m\wedge (\ell-m)$. 
Let us choose $D\subseteq I_0\setminus I_2$ with 
$a:=0\vee (2m-\ell)\vee (m-n+\ell-b)
=\card{D}\leq \card{I_0\setminus I_2}=m-b$.
Let $I_1=I_0\setminus(I_2\cup D)$, that is 
$\card{I_1}=m-a-b$. In particular, $I_1$, $I_2$, and $D$ 
are pairwise disjoint sets with
$I_0=I_1\cup I_2\cup D$, $K_0\cap D=\emptyset$,
$I_1\subseteq \set{n}\setminus(K_0\cup D)$, 
and $I_2\subseteq K_0$. 
We choose 
$J_1\subseteq \set{n}\setminus(K_0\cup D)$ and $J_2\subseteq K_0$
such that $I_1\subseteq J_1$, $I_2\subseteq J_2$, and 
$\card{J_1}=m-a=\card{J_2}$. 
We note that 
\begin{align*}
m-a
&=m-(0\vee (2m-\ell)\vee (m-n+\ell-b))\\
&\geq m-(0\vee (m+\ell-1-\ell)\vee (m-1))
=m-(m-1)=1, 
\end{align*}
giving $J_1\neq\emptyset\neq J_2$.
Set $J=K_0\cup J_1\cup D$, that is $\card{J}=\ell-m+m-a+a=\ell$. 
In particular, $(J,D,J_1,I_1,J_2,I_2)$ is admissible. 

Since $D\cup I_1\cup I_2= I_0$ and $J\setminus(D\cup J_1)=K_0$, 
we have 
$g(D)g(I_1)g(I_2)h(J\setminus (D\cup J_1))>0$. 
Using \eqref{e1543822}, we get 
\begin{align*}
g(J_1\setminus I_1)g(J_2\setminus I_2)h(J\setminus (D\cup J_2))>0
\end{align*}
and hence $g(J_1\cup J_2\cup D)>0$.

Let $K_1=J\setminus(D\cup J_2)$. Then $\card{K_1}=\ell-m$
and $h(K_1)>0$. 
Now it is possible to imitate the construction above with $K_1$ 
instead of $K_0$, which leads to a second admissible family
$(J,D,J_2,I_2,J_1,I_1)$, where 
the roles of $(J_1,I_1)$ and $(J_2,I_2)$ are interchanged. 

\item 
We now prove that $g(J)>0$.
If $J_2\setminus I_2\neq\emptyset$, we get from the above that 
$g(K_0\setminus I_2)>0$, 
since the elements of $J_2\setminus I_2$ can arbitrarily 
be chosen from $K_0\setminus I_2$. Hence, in this case, 
$g(J)=g(K_0\setminus I_2)g(I_2)g(J_1)g(D)>0$.
On the other hand, if $J_2\setminus I_2=\emptyset$, then 
$\card{I_1}=\card{J_2\setminus I_2}=0$ and
$I_1=\emptyset$, that is 
$J_1\setminus I_1=J_1\neq\emptyset$. 
Using the second admissible family, we obtain
$g(K_1\setminus I_1)>0$ and  
$g(J)=g(K_1\setminus I_1)g(I_1)g(J_2)g(D)>0$.

\item 
Let us now show that $g(\set{n})>0$. 
If $J_1\setminus I_1\neq\emptyset$, we obtain  
that $g(\set{n}\setminus (K_0\cup D\cup I_1))>0$, since 
the elements of $J_1\setminus I_1$ can arbitrarily 
be chosen from $\set{n}\setminus (K_0\cup D\cup I_1)$. Hence 
$g(\set{n})=g(\set{n}\setminus 
(K_0\cup D\cup I_1))g(K_0)g(D)g(I_1)>0$. 
On the other hand, if $J_1\setminus I_1=\emptyset$, then 
$\card{I_2}=\card{J_1\setminus I_1}=0$ and $I_2=\emptyset$, 
that is $J_2\setminus I_2=J_2\neq\emptyset$. 
Analogously to the above, we get that
$g(\set{n}\setminus (K_1\cup D\cup I_2))>0$ and
$g(\set{n})=g(\set{n}\setminus (K_1\cup D\cup I_2))g(K_1)g(D)g(I_2)
>0$. 

\item  \label{r84876.e}
Using \eqref{e1543822} and that $g(\set{n})>0$, 
it is easily shown that, 
if $h(K)>0$ for a set $K\subseteq\set{n}$
with $\card{K}=\ell-m$, then 
$h((K\cup\{j_1\})\setminus\{j_2\})>0$ 
for all $j_1\in \set{n}\setminus K$ and $j_2\in K$.
By iterating this procedure, we obtain 
that $h(K)>0$ for all $K\subseteq\set{n}$ with $\card{K}=\ell-m$.

\item 
Let us now show that $g_1=\dots=g_n$. For this purpose, 
we drop the notation from Part~(\ref{r84876.b}) and 
consider new admissible families $(J,D,J_1,I_1,J_2,I_2)$. 

Let $j_1,j_2\in\set{n}$ be arbitrary with $j_1\neq j_2$, 
$J_1=\{j_1\}$, 
$J_2=\{j_2\}$, $D\subseteq\set{n}\setminus\{j_1,j_2\}$ with
$a:=\card{D}=m-1$, $K_2\subseteq\set{n}\setminus (D\cup J_1)$ with
$J_2\subseteq K_2$, $\card{K_2}=\ell-m$, 
and $J=K_2\cup J_1\cup D$.

If $I_1=\emptyset$ and $I_2=\{j_2\}$, then 
$\card{I_1}=0=\card{J_1}-\card{I_2}\leq n-\ell$ and \eqref{e1543822} 
gives
\begin{align*}
g_{j_2}h(J\setminus(D\cup J_1))
&=g(I_1)g(I_2)h(J\setminus(D\cup J_1))\\
&=g(J_1\setminus I_1)g(J_2\setminus I_2)h(J\setminus(D\cup J_2))
=g_{j_1}h(J\setminus(D\cup J_2)).
\end{align*}
If $I_1=\{j_1\}$ and $I_2=\emptyset$, then 
$\card{I_1}=1=\card{J_1}-\card{I_2}\leq n-\ell$ 
and, again, \eqref{e1543822} implies that
\begin{align*}
g_{j_1}h(J\setminus(D\cup J_1))
&=g(I_1)g(I_2)h(J\setminus(D\cup J_1))\\
&=g(J_1\setminus I_1)g(J_2\setminus I_2)h(J\setminus(D\cup J_2))
=g_{j_2}h(J\setminus(D\cup J_2)).
\end{align*}
Hence $\frac{g_{j_2}}{g_{j_1}}=\frac{g_{j_1}}{g_{j_2}}$, 
that is $g_{j_1}=g_{j_2}$.

\item 
Similarly as in Part (\ref{r84876.e}), it can now be shown  
that $h(K)=h(K')$ for all $K,K'\subseteq\set{n}$ with 
$\card{K}=\card{K'}=\ell-m$. \qedhere
\end{enumerate}
\end{proof}
%%%%%%%%%%%%%%%%%%%%%%%%%%%%%%%%%%%%%%%%%%%%%%%%%%%%%%%%%%%%%%%%%%%%%%
\section{Remaining proofs}\label{s7296}
%%%%%%%%%%%%%%%%%%%%%%%%%%%%%%%%%%%%%%%%%%%%%%%%%%%%%%%%%%%%%%%%%%%%%%
\begin{proof}[Proof of Lemma \ref{l487698}] 
For $m\in\Zpl^d$ with $\vecsum{m}=n$ and 
$s\in\set{d}^n$ with $\weight(s)=m$, we have 
\begin{align*}
x^m=\prod_{k=1}^dx_k^{\weight_k(s)} 
=\prod_{k=1}^d\Bigl(\prod_{j=1}^n x_k^{\bbone_{\{k\}}(s_j)}\Bigr)
=\prod_{j=1}^n x_{s(j)}.
\end{align*}
Hence
\begin{align*}
\sum_{\newatop{m\in\Zpl^d}{\vecsum{m}=n}}
\sum_{\newatop{s\in\set{d}^n}{\weight(s)=m}}
\Bigl(\prod_{j=1}^n z_{j,s(j)}\Bigr)x^m
&=\sum_{\newatop{m\in\Zpl^d}{\vecsum{m}=n}}
\sum_{\newatop{s\in\set{d}^n}{\weight(s)=m}}
\prod_{j=1}^n (z_{j,s(j)}x_{s(j)})\\
&=\sum_{k\in\set{d}^n}\prod_{j=1}^n (z_{j,k(j)}x_{k(j)})
=\prod_{j=1}^n \Bigl(\sum_{k=1}^dz_{j,k}x_k\Bigr),
\end{align*}
which shows the first equality. Further, it follows that, 
for arbitrary $r\in\set{n}_{\neq}^n$, 
\begin{align*}
\sum_{\newatop{s\in\set{d}^n}{\weight(s)=m}}\prod_{j=1}^n z_{j,s(j)}
&=\sum_{\newatop{s\in\set{d}^n}{\weight(s)=m}}
\prod_{j=1}^n z_{r^{-1}(j),s(j)}
=\sum_{\newatop{s\in\set{d}^n}{\weight(s)=m}}
\prod_{j=1}^n z_{j,s(r_j)}.
\end{align*}
For $s,t\in\set{d}^n$ with 
$\weight(s)=\weight(t)=m$, there exists 
$\widetilde{r}\in\set{n}_{\neq}^n$
so that $t=s\circ\widetilde{r}$, giving
\begin{align*}
\sum_{r\in\set{n}_{\neq}^n}\prod_{j=1}^n z_{j,s(r_j)}
=\sum_{r\in\set{n}_{\neq}^n}\prod_{j=1}^n z_{j,s(\widetilde{r}(r_j))}
=\sum_{r\in\set{n}_{\neq}^n}\prod_{j=1}^n z_{j,t(r_j)}
=\per(Z').
\end{align*}
It is easily seen  that  
$\sum_{{s\in\set{d}^n:\,}{\weight(s)=m}}1=\frac{n!}{m!}$.
The second equality now follows from 
\begin{align*}
\sum_{\newatop{s\in\set{d}^n}{\weight(s)=m}}\prod_{j=1}^n z_{j,s(j)}
&=\frac{1}{n!}\sum_{r\in\set{n}_{\neq}^n} 
\sum_{\newatop{s\in\set{d}^n}{\weight(s)=m}}
\prod_{j=1}^n z_{j,s(r_j)}\\
&=\frac{1}{n!}\sum_{\newatop{s\in\set{d}^n}{\weight(s)=m}}
\sum_{r\in\set{n}_{\neq}^n} \prod_{j=1}^n z_{j,s(r_j)}
=\frac{1}{m!}\per(Z'). \qedhere
\end{align*}
\end{proof}
%%%%%%%%%%%%%%%%%%%%%%%%%%%%%%%%%%%%%%%%%%%%%%%%%%%%%%%%%%%%%%%%%%%%%%

\begin{proof}[Proof of Theorem \ref{c873765}]
Let $t\in\set{d}^n$ with $\weight(t)=m$,
$Z'=(z_{j,t(r)})\in\CC^{n\times n}$. 
For $k\in \set{d}$, let 
$M_k=\{r\in\set{n}\,|\,t(r)=k\}$. Then $M_1,\dots,M_d\subseteq\set{n}$
are pairwise disjoint with $\bigcup_{k=1}^dM_k=\set{n}$. Further
$\card{M_k}=m_k$ for all $k\in\set{d}$.
Lemma \ref{l487698} and Corollary \ref{c639860} imply 
the assertion.  
\end{proof}
%%%%%%%%%%%%%%%%%%%%%%%%%%%%%%%%%%%%%%%%%%%%%%%%%%%%%%%%%%%%%%%%%%%%%%
\begin{lemma}\label{l2757}
Let $n\in\NN$, $Z=(z_{j,r})\in(\CC\setminus\{0\})^{n\times n}$, and
$\abs{Z}=(\abs{z_{j,r}})\in(0,\infty)^{n\times n}$. Then we have 
$\abs{\per(Z)}=\per(\abs{Z})$ if and only if there are numbers 
$\xi_j,\zeta_r\in\CC$ for all $j,r\in\set{n}$ 
such that $\abs{\xi_j}=\abs{\zeta_r}=1$ and
$z_{j,r}=\xi_{j}\zeta_r \abs{z_{j,r}}$ for all $j,r\in\set{n}$. 
\end{lemma}
%%%%%%%%%%%%%%%%%%%%%%%%%%%%%%%%%%%%%%%%%%%%%%%%%%%%%%%%%%%%%%%%%%%%%%
\begin{proof}
See \citet[page 8]{MR2275869}.
\end{proof}

%%%%%%%%%%%%%%%%%%%%%%%%%%%%%%%%%%%%%%%%%%%%%%%%%%%%%%%%%%%%%%%%%%%%%%
\begin{proof}[Proof of the necessity of the condition
given in  Example \ref{ex167598}(\ref{ex167598.b}),(iii)] 
\rule{0pt}{0pt} \\
Let us assume that $e$, $f$, $g$, and $h$ are all non-zero vectors. 
Inequality \eqref{eq357865} can be shown by using  
Example \ref{ex167598}(\ref{ex167598.a}) applied to the transpose of 
the matrix under consideration and the Cauchy-Schwarz inequality: 
\begin{align} 
\abs{\per(Z)}^2
&=\abs{z_{1,3}\per(Z[\{2,3\},\set{2}])
+z_{2,3}\per(Z[\{1,3\},\set{2}])
+z_{3,3}\per(Z[\set{2},\set{2}])}^2\nonumber\\
&\leq\bigl(\abs{z_{1,3}\per(Z[\{2,3\},\set{2}])}
+\abs{z_{2,3}\per(Z[\{1,3\},\set{2}])}
+\abs{z_{3,3}\per(Z[\set{2},\set{2}])}\bigr)^2\nonumber\\
&\leq\bigl(\abs{z_{1,3}}\norm{f}\norm{g}
+\abs{z_{2,3}}\norm{e}\norm{g}
+\abs{z_{3,3}}\norm{e}\norm{f}\bigr)^2\label{eq3645764}\\
&\leq (\norm{e}^2\norm{f}^2+\norm{e}^2\norm{g}^2+\norm{f}^2\norm{g}^2)
\norm{h}^2.\label{eq3645765}
\end{align}
We assume that, in \eqref{eq357865}, equality holds. Therefore, in 
the chain above, equality holds. In particular, we have
$\abs{z_{1,1}z_{2,1}}=\abs{z_{1,2}z_{2,2}}$,
$\abs{z_{1,1}z_{3,1}}=\abs{z_{1,2}z_{3,2}}$, and
$\abs{z_{2,1}z_{3,1}}=\abs{z_{2,2}z_{3,2}}$.

Let us verify that $Z\in(\CC\setminus\{0\})^{3\times 3}$.  
We first show that $z_{1,1}\neq0$. Suppose that $z_{1,1}=0$. 
Since $e\neq 0$, we have $z_{1,2}\neq0$ and, in turn, 
$z_{2,2}=0$, because $0=\abs{z_{1,1}z_{2,1}}=\abs{z_{1,2}z_{2,2}}$. 
Since $f\neq 0$, we have $z_{2,1}\neq0$ and, in turn, 
$z_{3,1}=0$, because $\abs{z_{2,1}z_{3,1}}=\abs{z_{2,2}z_{3,2}}=0$. 
But from $g\neq0$, it follows that $z_{3,2}\neq 0$, that is
$0=\abs{z_{1,1}z_{3,1}}=\abs{z_{1,2}z_{3,2}}\neq0$, which is a 
contradiction. Hence $z_{1,1}\neq 0$. Similarly, one can show that 
$z_{2,1},z_{3,1},z_{1,2},z_{2,2},z_{3,2},$ are all non-zero.  
Since $h\neq0$ and, in \eqref{eq3645765}, equality holds, 
a number $x\in(0,\infty)$ exists such that
\begin{align}\label{e782566}
(\abs{z_{1,3}},\abs{z_{2,3}},\abs{z_{3,3}})
=\frac{x}{2}(\norm{f}\norm{g},\norm{e}\norm{g},\norm{e}\norm{f}).
\end{align}
The numbers $z_{1,3},z_{2,3},z_{3,3}$ are non-zero, because
$e$, $f$, and $g$ are non-zero.  
Therefore, we have $Z\in(\CC\setminus\{0\})^{3\times 3}$. 

From the equality in \eqref{eq357865}, it follows that 
$\abs{\per(Z)}=\per(\abs{Z})$ and hence 
there are numbers $\xi_j,\zeta_r\in\CC$ for $j,r\in\set{3}$
such that $\abs{\xi_j}=\abs{\zeta_r}=1$ and
$z_{j,r}=\xi_{j}\zeta_r \abs{z_{j,r}}$ for all $j,r\in\set{3}$, see 
Lemma \ref{l2757}. Without loss of generality, we may now assume that 
$Z\in(0,\infty)^{3\times 3}$. 

Since in \eqref{eq3645764} equality holds, we have
$z_{1,1}z_{2,1}=z_{1,2}z_{2,2}$,
$z_{1,1}z_{3,1}=z_{1,2}z_{3,2}$, and
$z_{2,1}z_{3,1}=z_{2,2}z_{3,2}$.
In particular, 
$(z_{1,1}z_{2,1})(z_{1,2}z_{3,2})
=(z_{1,2}z_{2,2})(z_{1,1}z_{3,1})$,
i.e.\ $z_{2,1}z_{3,2}=z_{3,1}z_{2,2}$. 
Similarly, we get that
$z_{1,1}z_{3,2}=z_{3,1}z_{1,2}$
and $z_{2,1}z_{1,2}=z_{1,1}z_{2,2}$. 
Therefore 
$(z_{1,1}z_{3,1})(z_{3,1}z_{1,2})=(z_{1,2}z_{3,2})(z_{1,1}z_{3,2})$, 
i.e.\ $z_{3,1}=z_{3,2}$. 
Hence
$z_{2,1}z_{3,2}=z_{2,1}z_{3,1}=z_{2,2}z_{3,2}$, i.e.\ 
$z_{2,1}=z_{2,2}$. Further
$z_{1,1}z_{2,2}=z_{1,1}z_{2,1}=z_{1,2}z_{2,2}$, i.e.\ 
$z_{1,1}=z_{1,2}$.
Consequently, $z_{j,1}=z_{j,2}$ for all $j\in\set{3}$.

In particular, $\norm{e}^2=2z_{1,1}^2$, $\norm{f}^2=2z_{2,1}^2$, and
$\norm{g}=2z_{3,1}^2$. Now, \eqref{e782566} implies that 
$z_{1,3}=\frac{x}{2}\norm{f}\norm{g}=xz_{2,1}z_{3,1}$, 
$z_{2,3}=\frac{x}{2}\norm{e}\norm{g}=xz_{1,1}z_{3,1}$, and
$z_{3,3}=\frac{x}{2}\norm{e}\norm{f}=xz_{1,1}z_{2,1}$.
Hence $z_{j,3}=x\prod_{i\in\set{3}\setminus\{j\}}z_{i,1}$
for all $j\in\set{3}$.
\end{proof}

%%%%%%%%%%%%%%%%%%%%%%%%%%%%%%%%%%%%%%%%%%%%%%%%%%%%%%%%%%%%%%%%%%%%%%
\section*{Acknowledgments}
\noindent
The author thanks the reviewer for his remarks, which helped to 
improve a previous version of this paper. 

%%%%%%%%%%%%%%%%%%%%%%%%%%%%%%%%%%%%%%%%%%%%%%%%%%%%%%%%%%%%%%%%%%%%%%
\small 
%\raggedbottom % inserted to avoid Underfull \vbox (badness 10000)
\linespread{1}\bibsep7pt
\selectfont
\bibliographystyle{elsarticle-harv} 
\bibliography{ghti_72_rev_15}

%%%%%%%%%%%%%%%%%%%%%%%%%%%%%%%%%%%%%%%%%%%%%%%%%%%%%%%%%%%%%%%%%%%%%%
\end{document}